\documentclass[10pt]{article}
\usepackage{amsfonts}
\usepackage{amsthm}
\usepackage{amssymb}
\usepackage{mathrsfs}
\usepackage{geometry}
\usepackage{graphicx}
\usepackage[utf8]{inputenc}
\usepackage[T1]{fontenc}
\usepackage{amsmath}
\usepackage{float}

\usepackage{authblk}

\usepackage{hyperref}
\hypersetup{backref, pdfpagemode=FullScreen, colorlinks=true,
	citecolor=magenta, linkcolor=cyan, urlcolor=blue}

\usepackage{mathtools}
\mathtoolsset{showonlyrefs}

\usepackage{version}

\theoremstyle{change}
\newtheorem{theo}{Theorem}[section]
\newtheorem*{theo*}{Theorem}

\newtheorem{prop}[theo]{Proposition}

\begin{document}
	\title{A note on the maximal perimeter and maximal width of a convex small polygon}
	\author{Fei Xue\footnote{Fei Xue: 05429@njnu.edu.cn\\School of Mathematical Sciences, Nanjing Normal University, No.1 Wenyuan Road Qixia District, Nanjing, P.R.China 210046}, Yanlu Lian\footnote{Yanlu Lian: yanlu\_lian@tju.edu.cn\\Center for Applied Mathematics, Tianjin University, Tianjin, P.R.China 300354}, Jun Wang\footnote{Jun Wang: 18463756956@163.com\\Center for Applied Mathematics, Tianjin University, Tianjin, P.R.China 300354}, Yuqin Zhang\footnote{Yuqin Zhang: yuqinzhang@tju.edu.cn\\School of Mathematics, Tianjin University, Tianjin, P.R.China 300072}}
	\maketitle
	\begin{abstract}
		The polygon $P$ is small if its diameter equals one. When $n=2^s$, it is still an open problem to find the maximum perimeter or the maximum width of a small $n$-gon. Motivated by Bingane's series of works, we improve the lower bounds for the maximum perimeter and the maximum width.
	\end{abstract}
	
\section{Introduction}

Let $P$ be a convex polygon. The diameter of $P$ is the maximum distance between pairs of its vertices. The polygon $P$ is small if its diameter equals one. The diameter graph of a small polygon is defined as the graph with the vertices of the polygon, and an edge between two vertices exists only if the distance between these vertices equals one. Diameter graphs of some convex small polygons are represented in Figure 1, Figure 2, and Figure 3. The blue lines are sides, and the black lines are edges. The height associated to a side of $P$ is defined as the maximum distance between a vertex and the line containing the side. The minimum height for all sides is the width of the polygon $P$.

\begin{figure}[htbp]
\begin{minipage}[htbp]{0.55\linewidth}
\centering
\includegraphics[height=6cm,width=6cm]{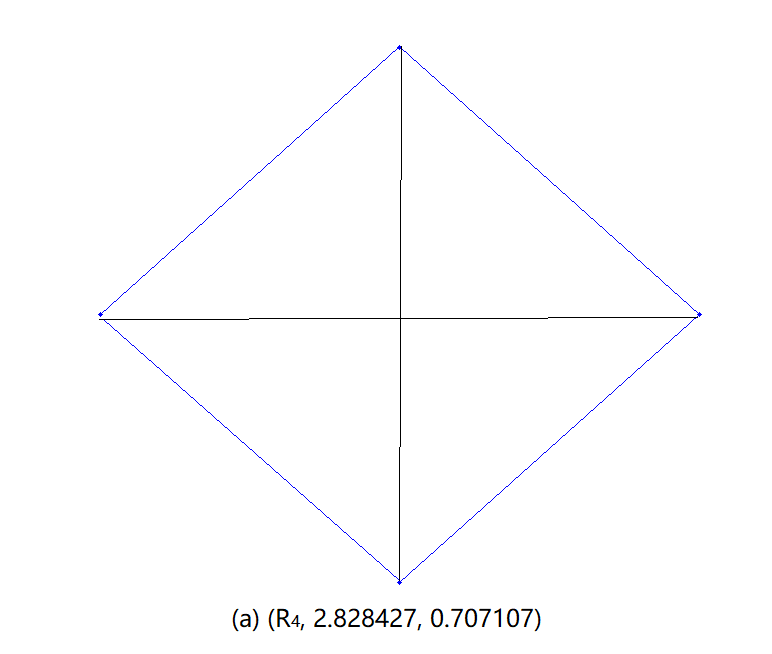}
\end{minipage}%
\hfill
\begin{minipage}[htbp]{0.55\linewidth}
\centering
\includegraphics[height=6cm,width=6cm]{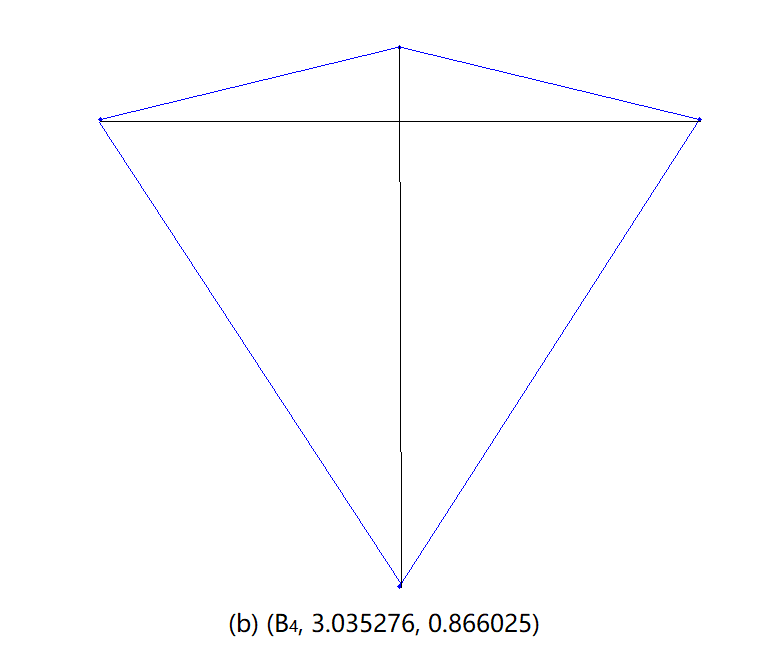}
\end{minipage}
\caption{Two convex small $4$-gon $(P_4, L(P_4), W(P_4)):$ (a) Regular $4$-gon; (b) $4$-gon of maximal perimeter and maximal width.}
\end{figure}
For a given integer $n\geq 3$, we are looking forward for the maximal perimeter and the maximal width of
a convex small $n$-gon. This problem is firstly studied by Reinhardt~\cite{Reinhardt} and Datta~\cite{Datta}.
They showed that, for an integer $n\geq 3$, the value $2n\sin\frac{\pi}{2n}$ is an upper bound
on the perimeter $L(P_n)$ of a convex small $n$-gon $P$. For the width part, it was first investigated
by A. Bezdek and F. Fodor~\cite{Bezdek} that the value $\cos\frac{\pi}{2n}$ is an upper bound on its width $W(P_n)$.
For $n=2^s$ with $s\geq 2$, Bingane~\cite{Bingane1} constructed a family of convex small $n$-gon,
whose perimeters and widths differ from the upper bounds $2n \sin\frac{\pi}{2n}$ and $\cos\frac{\pi}{2n}$ by just $O(\frac{1}{n^6})$ and $O(\frac{1}{n^4})$, respectively. For $n=2^s$ with $s\geq 2$, Bingane~\cite{Bingane2} constructed a family of convex small n-gon, and showed that the perimeters and the widths obtained cannot be improved for large $n$ by more than $\frac{a}{n^6}$ and $\frac{b}{n^4}$ respectively, for certain positive constants $a$ and $b$. For $n=2^s$ with $s\geq 4$, Bingane~\cite{Bingane4} further tighten lower bounds on the maximal perimeter and the maximal width.
	
Inspired by Bingane's works, we furtherly study the construction of small polygons. Our main result is:

\begin{theo}\label{theo:main}
	Suppose $n=2^s$ with integer $s\geq 4$. Let $\overline{L}_n=2n\sin\frac{\pi}{2n}$ denote an upper bound on the perimeter $L(P_n)$ of a convex small $n$-gon $P_n$, and let $\overline{W}_n=\cos\frac{\pi}{2n}$ denote an upper bound on its width $W(P_n)$. Let
	$$\sigma_n=\min_{b_k=\pm 1}\frac{\mid\sum_{k=0}^{\frac{n}{4}-1}b_k\sin\left(\left(2k+1\right)\cdot\frac{\pi}{n}\right)\mid}{\sum_{k=0}^{\frac{n}{4}-1}\cos\left(\left(2k+1\right)\cdot\frac{\pi}{n}\right)},$$
and let $\delta_n\in \left(0, \frac{\pi}{n}\right)$ be the solution of
	$$\sigma_n\left(\cos x-\cos\frac{\pi}{n}\right)=\sin x,$$
    i.e., $$\delta_n=\arccos\left(\frac{\sigma_n}{\sqrt{1+\sigma_n^2}}\cos{\frac{\pi}{n}}\right)-\arccos\frac{\sigma_n}{\sqrt{1+\sigma_n^2}}.$$
    Then there exists a convex small $n$-gon $D_n$ such that
      $$L(D_n)= 2n\sin\frac{\pi}{2n}\cos\frac{\delta_n}{2},$$
   $$W(D_n)= \cos\left(\frac{\pi}{2n}+\frac{\delta_n}{2}\right),$$
and
$$\overline{L}_n-L(D_n)\leq O\left(\frac{\pi^{\log_2n^2}}{n^{\log_2n+5}}\right).$$
$$\overline{W}_n-W(D_n)\leq O\left(\frac{\pi^{\log_2n^2}}{n^{\log_2n+7}}\right).$$

\end{theo}

	This paper is organized as follows. In Section 2, we briefly discuss the basic properties of small polygons
	of maximal perimeter and of maximal width, and then describe the corresponding optimization problem.
	In Section 3, we study the optimization problem and construct the candidates mentioned in Theorem~\ref{theo:main},
	and then prove Theorem~\ref{theo:main}.
	
\section{Perimeters and widths of convex small polygons}

\subsection{History}

Let $L(P)$ denote the perimeter of a polygon $P$ and $W(P)$ its width. It is pointed out by Reinhardt~\cite{Reinhardt} in 1922, and later by Datta~\cite{Datta} in 1997, that the regular small $n$-gon $R_n$ attains
the maximal perimeter $L(R_n)=2n\sin\frac{\pi}{2n}$ and the maximal width $W(R_n)=\cos\frac{\pi}{2n}$ for $n$ odd, and the maximal perimeter $L(R_n)=n\sin\frac{\pi}{n}$ and the maximal width $W(R_n)=\cos\frac{\pi}{n}$ for $n$ even.
When $n$ has an odd factor $m$, consider the family of convex equilateral small $n$-gon constructed as follows:

1. Transform the regular small $m$-gon $R_m$ into a Reuleaux $m$-gon by replacing each edge by a circle’s arc passing through its end vertices and centered at the opposite vertex;

2. Add at regular intervals $\frac{n}{m}-1$ vertices within each arc;

3. Take the convex hull of all vertices. These $n$-gon are denoted by $R_{m,n}$ and
$$L(R_{m,n})=2n\sin\frac{\pi}{2n},$$
$$W(R_{m,n})=\cos\frac{\pi}{2n}.$$

\begin{figure}[htbp]
\begin{minipage}[htbp]{0.55\linewidth}
\centering
\includegraphics[height=6cm,width=6cm]{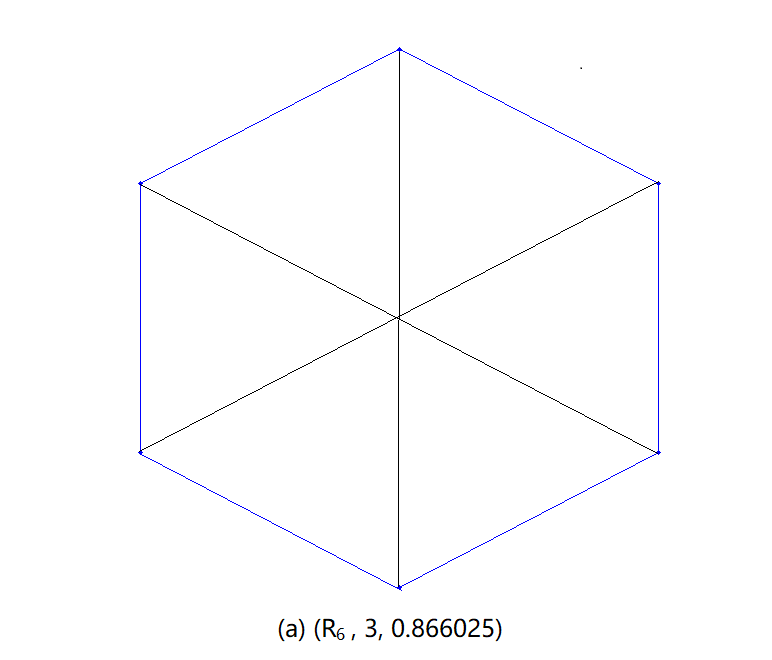}
\end{minipage}%
\hfill
\begin{minipage}[htbp]{0.55\linewidth}
\centering
\includegraphics[height=6cm,width=6cm]{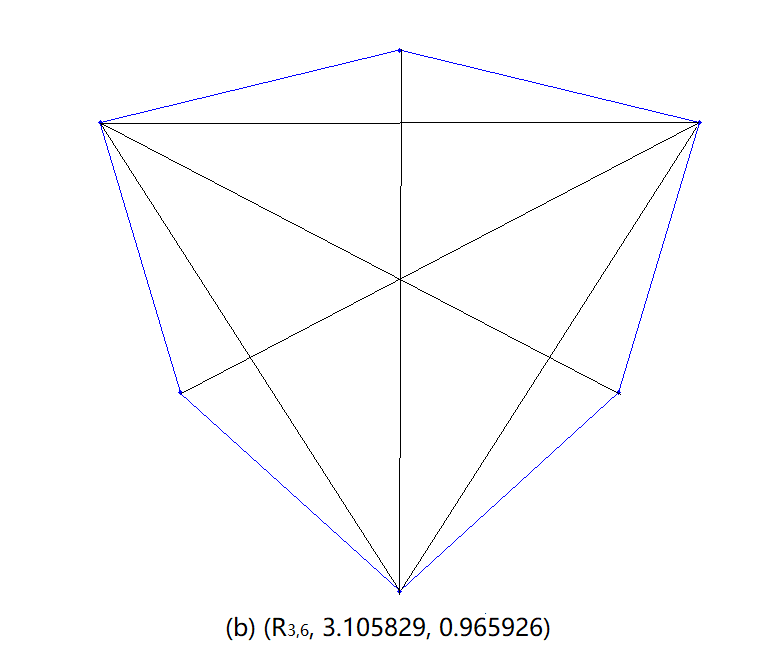}
\end{minipage}
\caption{Two convex small $6$-gon $(P_6, L(P_6), W(P_6)):$ (a) Regular $6$-gon; (b) Reinhardt $6$-gon.}
\end{figure}

\begin{figure}[htbp]
\begin{minipage}[htbp]{0.33\linewidth}
\centering
\includegraphics[height=4cm,width=5cm]{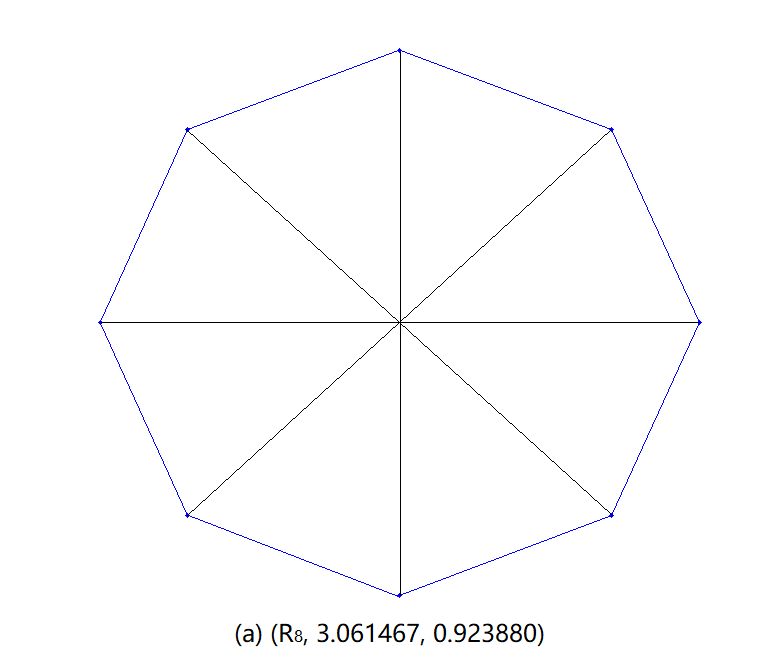}
\end{minipage}%
\hfill
\begin{minipage}[htbp]{0.33\linewidth}
\centering
\includegraphics[height=4cm,width=5cm]{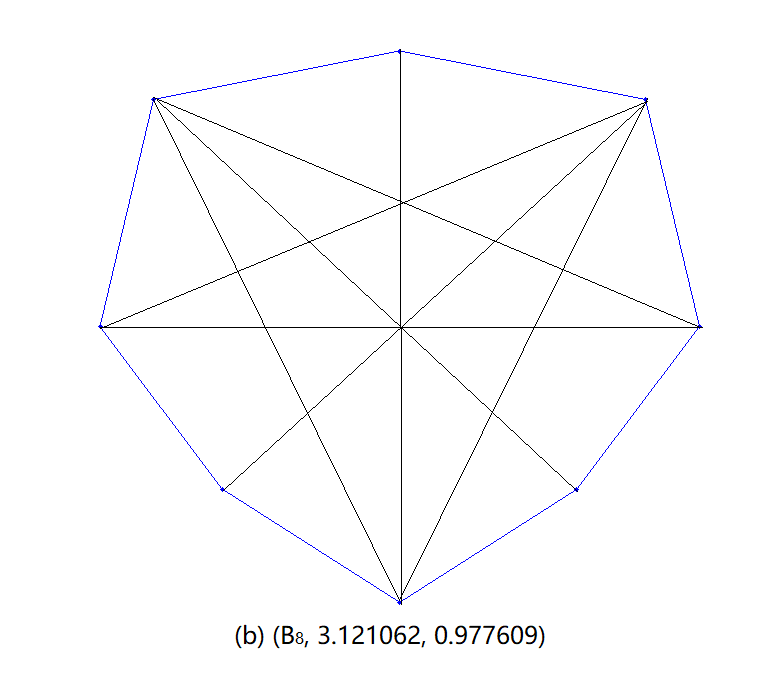}
\end{minipage}
\hfill
\begin{minipage}[htbp]{0.33\linewidth}
\centering
\includegraphics[height=4cm,width=5cm]{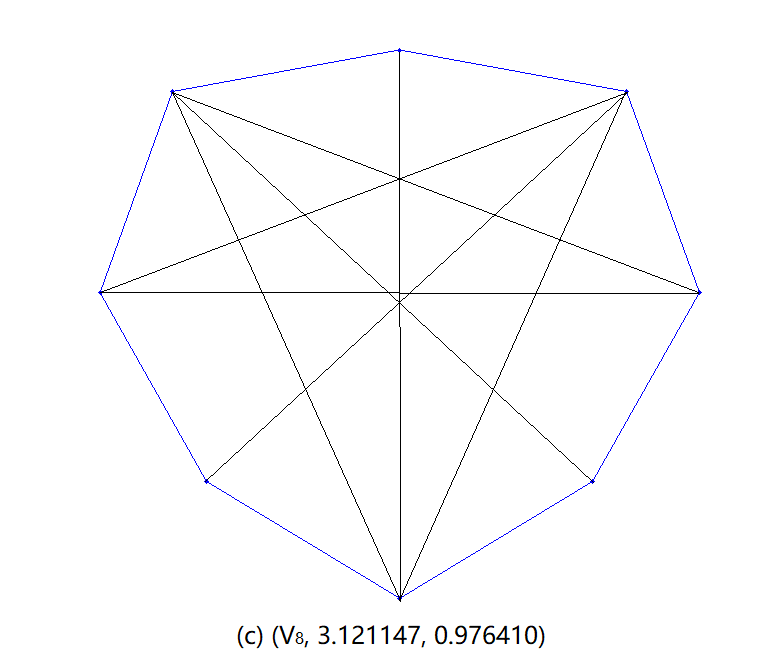}
\end{minipage}
\caption{Three convex small $8$-gon $(P_8, L(P_8), W(P_8)):$ (a) Regular $8$-gon; (b) $8$-gon of maximal width; (c) $8$-gon of maximal perimeter.}
\end{figure}
\begin{theo*}[Reinhardt~\cite{Reinhardt}, Datta~\cite{Datta}]
For all $n\geq 3$, let $L_n^{*}$ denote the maximal perimeter among all convex small $n$-gon and $\overline{L}_n= 2n\sin\frac{\pi}{2n}$.

•When $n$ has an odd factor $m$, $L_n^{*}=\overline{L}_n$ is achieved by finitely many equilateral $n$-gon \cite{Hare1,Hare2,Mossinghoff3}, including $R_{m,n}$. The optimal $n$-gon $R_{m,n}$ is unique if $m$ is prime and $\frac{n}{m}\leq 2.$

•When $n= 2^s$ with $s\geq 2$, $L(R_n)< L_n^{*}<\overline{L}_n$.
 \end{theo*}

 When $n=2^s$, the maximal perimeter $L_n^{*}$ is only known for $s\leq3$. Tamvakis~\cite{Tamvakis} proved that $L_4^{*}=2+\sqrt{6}-\sqrt{2}$, and this value is only achieved by $B_4$, represented in Figure 1b. Audet, Hansen and Messine~\cite{Audet1} proved that $L_8^{*}= 3.1211471340\dots$, which is only achieved by $V_8$, represented in Figure 3c.

For the width case:

\begin{theo*}[Bezdek and Fodor~\cite{Bezdek}]
For all $n\geq 3$, let $W_n^{*}$ denote the maximal width among all convex small $n$-gon and let $\overline{W}_n= \cos\frac{\pi}{2n}$.

•When $n$ has an odd factor, $W_n^{*}=\overline{W}_n$ is achieved by a convex small $n$-gon with maximal perimeter $L_n^{*}=\overline{L}_n$.

•When $n=2^s$ with integer $s\geq 2$, $W(R_n)< W_n^{*}<\overline{W}_n$.
\end{theo*}

When $n=2^s$, similar with the perimeter case, the maximal width $W_n^{*}$ is known for $s\leq 3$. Bezdek and Fodor~\cite{Bezdek} proved that $W_4^{*}=12\sqrt{3}$, and this value is achieved by infinitely many convex small $4$-gon, including that of maximal perimeter $B_4$. Audet, Hansen, Messine and Ninin~\cite{Audet1} found that $W_8^{*}=14\sqrt{10+2\sqrt{7}}$, which is also achieved by infinitely many convex small $8$-gon, including $B_8$ represented in Figure 3b.

In 2006, Mossinghoff~\cite{Mossinghoff1} conjectured that, when $n=2^s$ and $s\geq 2,$ the diameter graph of a convex small $n$-gon of maximal perimeter has a cycle of length $\frac{n}{2}+1$ plus $\frac{n}{2}-1$ additional pendant edges, and that is verified for $s=2$ and $s=3$. However, the conjecture is no longer true for $s\geq 4$ as the perimeter of $D_n$ exceeds that of the optimal $n$-gon obtained by Mossinghoff. Thus, it remains to study the case $n=2^s$, for $s\geq 4$.
Bingane~\cite{Bingane1} constructed a family of convex small $n$-gon $B_n$.

For $n=2^s$ with $s\geq 2$,
$$L(B_n)=2n\sin\frac{\pi}{2n}\cos\left(\frac{\pi}{2n}-\frac{1}{2}\arcsin\left(\frac{1}{2}\sin{\frac{2\pi}{n}}\right)\right),$$
$$W(B_n)=\cos\left(\frac{\pi}{n}-\frac{1}{2}\arcsin\left(\frac{1}{2}\sin{\frac{2\pi}{n}}\right)\right),$$
and
$$\overline{L}_n-L(B_n)=\frac{\pi^7}{32n^6}+O\left(\frac{1}{n^8}\right),$$
$$\overline{W}_n-W(B_n)=\frac{\pi^4}{8n^4}+O\left(\frac{1}{n^6}\right).$$
By contrast,
$$\overline{L}_n-L(R_n)=\frac{\pi^3}{8n^2}+O\left(\frac{1}{n^4}\right),$$
$$\overline{W}_n-W(R_n)=\frac{3\pi^2}{8n^2}+O\left(\frac{1}{n^4}\right)$$
for all even $n\geq 4$. Note that $L(B_n)=L_4^{*}=W_4^{*}$, and $W(B_8)=W_8^{*}.$ Recently, Bingane~\cite{Bingane4} constructed a family of convex small $n$-gon for any $n=2^s$ where $s\geq 4$, showing that:

\begin{theo*}[Bingane~\cite{Bingane4}]
Suppose $n=2^s$ with integer $s\geq 4$. Let $\overline{L}_n=2n\sin{\frac{\pi}{2n}}$ denote an upper bound on the perimeter $L(P_n)$ of a convex small $n$-gon $P_n$, and $\overline{W}_n=\cos\frac{\pi}{2n}$ denote an upper bound on its width $W(P_n)$. Then there exists a convex small $n$-gon $C_n$ such that
$$L(C_n)=2n\sin\frac{\pi}{2n}\cos\left(\frac{1}{2}\arctan\left(\tan\frac{2\pi}{n}\tan\frac{\pi}{n}\right)
-\frac{1}{2}\arcsin\left(\frac{\sin\left(\frac{2\pi}{n}\right)\sin\left(\frac{\pi}{n}\right)}
{\sqrt{4\sin^2\left(\frac{\pi}{n}\right)+\cos\left(\frac{4\pi}{n}\right)}}\right)\right),$$

$$W(C_n)=\cos\left(\frac{\pi}{2n}+\frac{1}{2}\arctan\left(\tan\frac{2\pi}{n}\tan\frac{\pi}{n}\right)
-\frac{1}{2}\arcsin\left(\frac{\sin\left(\frac{2\pi}{n}\right)\sin\left(\frac{\pi}{n}\right)}
{\sqrt{4\sin^2\left(\frac{\pi}{n}\right)+\cos\left(\frac{4\pi}{n}\right)}}\right)\right),$$
and
$$\overline{L}_n-L\left(C_n\right)=\frac{\pi^9}{8n^8}+O\left(\frac{1}{n^{10}}\right),$$
$$\overline{W}_n-W\left(C_n\right)=\frac{\pi^5}{4n^5}+O\left(\frac{1}{n^{7}}\right).$$
\end{theo*}

Now we briefly introduce how Bingane~\cite{Bingane4} constructed his small polygons:

For any $n=2^s$ where $s\geq 4$ is an integer, consider a convex small $n$-gon $C_n$ having the following diameter graph: a $\frac{3n}{4}-1$-length cycle $v_0-v_1-\dots-v_k-\dots-v_{\frac{3n}{8}-1}-v_{\frac{3n}{8}}-\dots-v_{\frac{3n}{4}-k-1}-\dots-v_{\frac{3n}{4}-2}-v_0$
plus $\frac{n}{4}+1$ pendant edges $v_0-v_{\frac{3n}{4}}-1$,$v_{3j-2}-v_{\frac{3n}{4}+j+1},j=1,2,\dots, \frac{n}{4}$. Assume that $P_n$ has the edge $v_0-v_{\frac{3n}{4}-1}$ as axis of symmetry, and for all $j=1,2,\dots,\frac{n}{4},$ the pendant edge $v_{3j-2}-v_{\frac{3n}{4}+j-1}$ bissects the angle form at the vertex
$v_{3j-2}$ by the edge $v_{3j-2}-v_{3j-1}$ and the edge $v_{3j-2}-v_{3j-3}.$

For $k=0,1,\dots,\frac{3n}{8}-1$, let $c_k=2$ if $k=3k-2$, and $c_k=1$ otherwise. Then let $c_0\alpha_0$ denote the angle form at the vertex $v_0$ by the edge $v_0-v_1$ and the edge $v_0-v_{\frac{3n}{4}-1}$, and for all $k=1,2,\dots,\frac{3n}{8}-1$, $c_k\alpha_k$ the angle formed at the vertex $v_k$ by the edge $v_k-v_{k+1}$ and the edge $v_k-v_{k-1}$. Suppose $\alpha_k=\frac{\pi}{n}+(-1)^k\delta$ with $\mid \delta\mid<\frac{\pi}{n}$ for all $k=0,1,\dots,\frac{3n}{8}-1.$ Then, by taking the convex hull of all vertices, they constructed a convex small polygon $C_n$.
\subsection{Discipline to find small polygons }\label{section}

Let $n=2^s$ where $s\geq4$ is an integer. Consider the following convex small $n$-gon $P_n$. Let the vertices of $P_n$ be $v_0-v_2-\dots-v_{\frac{n}{2}-2}-v_{\frac{n}{2}}-v_{\frac{n}{2}+2}-\dots-v_{n-2}-v_1-v_3-\dots-v_{\frac{n}{2}+1}-\dots-v_{n-1}-v_0,$ in the clockwise order, and let $v_{n}=v_0,v_{n+1}=v_1$.

Let $v_{2k},v_{2k+1},v_{2k+2},v_{2k+3}$ be four vertices of $P_n$, and $v_{2k}-v_{2k+1}$ and $v_{2k+2}-v_{2k+3}$ are two edges of $P_n$. Consider the following diameter pattern:

Pattern I: $v_{2k}v_{2k+1},v_{2k+1}v_{2k+2},v_{2k+2}v_{2k+3}$ are three adjacent diameters, and the diameter changes in the way of $v_{2k}v_{2k+1}\rightarrow v_{2k+1}v_{2k+2}\rightarrow v_{2k+2}v_{2k+3}.$ (See Figure 4a)

Pattern II: $v_{2k}v_{2k+1},v_{2k}v_{2k+3},v_{2k+3}v_{2k+2}$ are three adjacent diameters, and the diameter changes in the way of $v_{2k}v_{2k+1}\rightarrow v_{2k}v_{2k+3}\rightarrow v_{2k+3}v_{2k+2}.$ (See Figure 4b)

\begin{figure}[H]
\begin{minipage}[htbp]{0.6\linewidth}
\centering
\includegraphics[height=6cm,width=6cm]{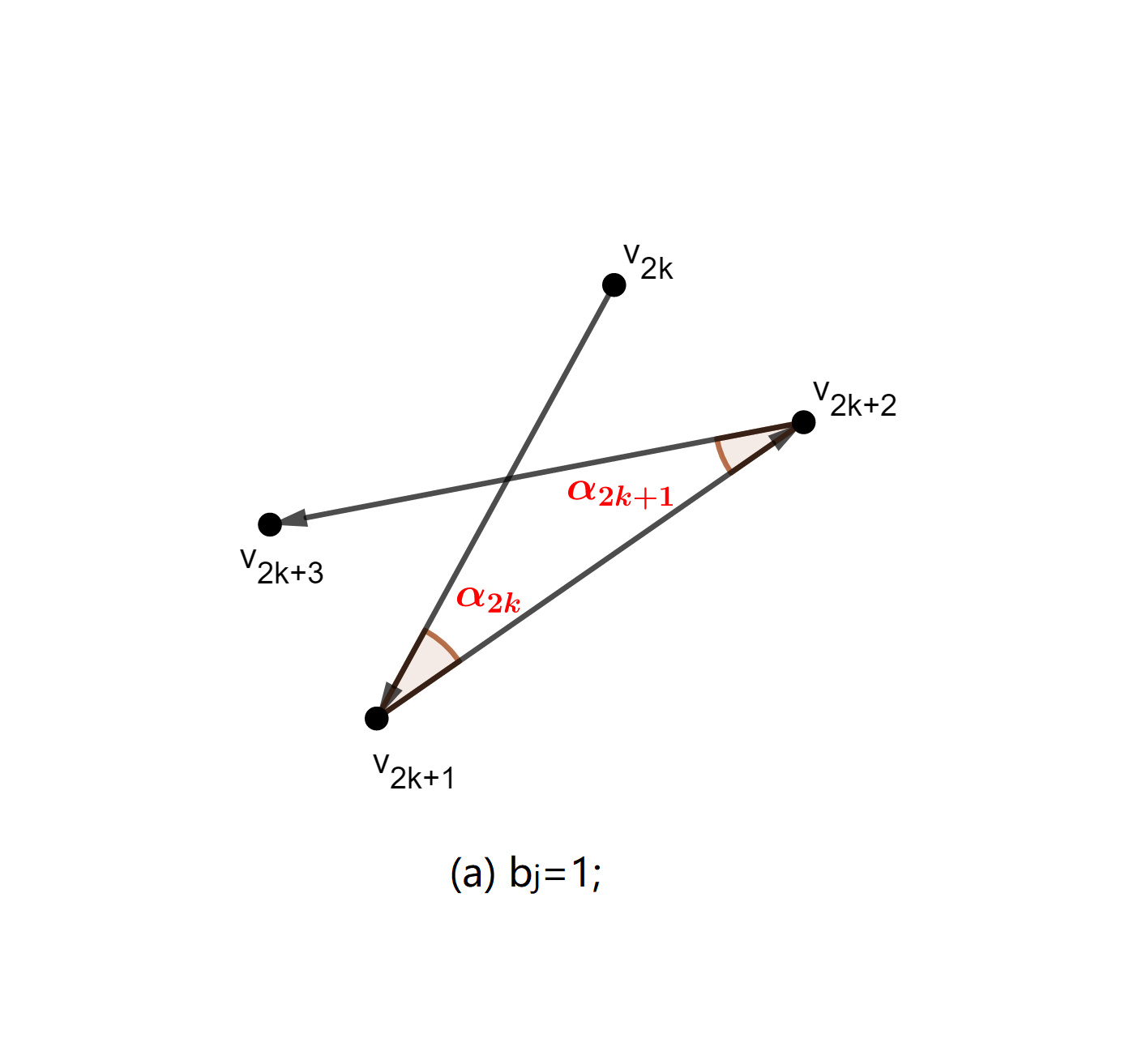}
\end{minipage}%
\hfill
\begin{minipage}[htbp]{0.6\linewidth}
\centering
\includegraphics[height=6cm,width=6cm]{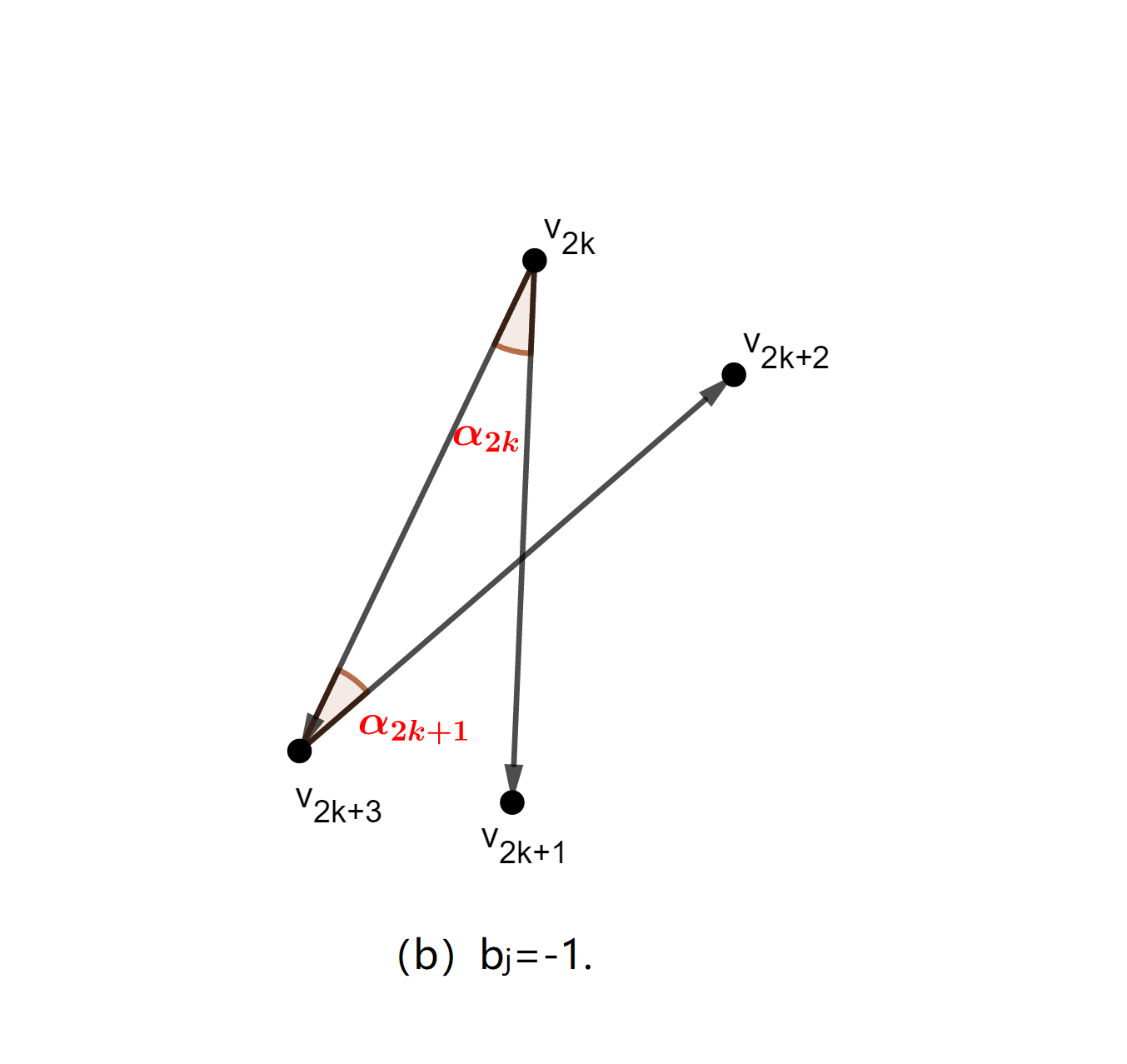}
\end{minipage}
\caption{three adjacent diameters formed by four vertices $v_{2k},v_{2k+1},v_{2k+2},v_{2k+3}$.}
\end{figure}

For all $j=0,1,\dots,\frac{n}{2}-1,$ let

$$b_j=
 \begin{cases}
 +1,&\text{if $v_{2k}v_{2k+1}\rightarrow v_{2k+1}v_{2k+2}\rightarrow v_{2k+2}v_{2k+3}$},\\
 -1,&\text{if $v_{2k}v_{2k+1}\rightarrow v_{2k}v_{2k+3}\rightarrow v_{2k+3}v_{2k+2}$.}
 \end{cases}$$

\begin{figure}[htbp]
\centering
\includegraphics[height=7cm,width=7cm]{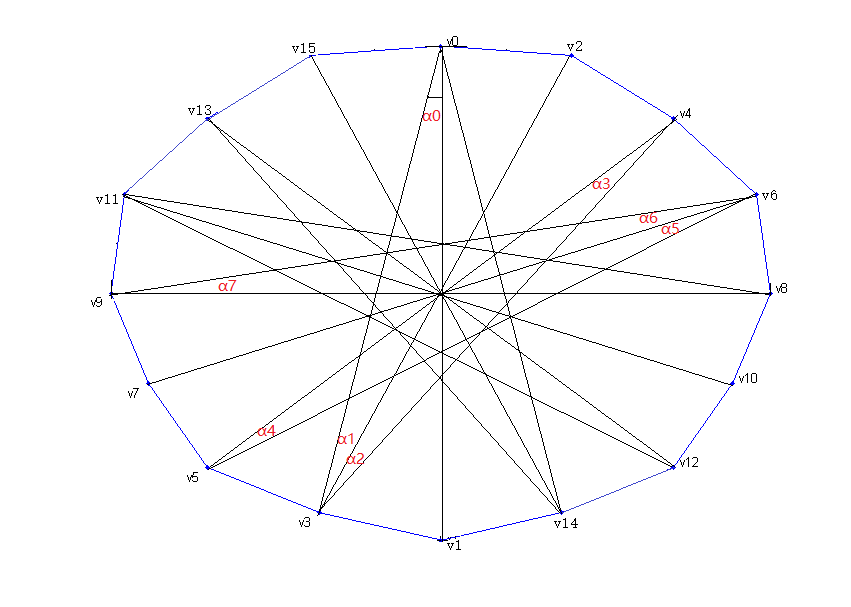}
\caption{Definition of variables: Example of $n=16$ vertices with $(b_1,b_2,b_3,b_4)=(-1,1,1,-1)$}
\end{figure}

Moreover, we assume that $P_n$ is symmetric with respect to the edge $v_0-v_1$, and let $\alpha_{0}$ denote the first rotation angle formed by the vertices $v_0,v_1,v_2,v_3$. Since the family of convex equilateral small $n$-gon constructed as the combination of $I$ and $II$, $\alpha_0$ can be formed by the edge $v_0-v_1$ and the edge $v_1-v_2$, or by the edge $v_0-v_1$ and the edge $v_0-v_3$. For all $k=0,1,\dots,\frac{n}{2}-1$, let $\alpha_{2k}$ denote the first rotation angle formed by the four vertices $v_{2k},v_{2k+1},v_{2k+2},v_{2k+3}$, and let $\alpha_{2k+1}$ denote the second rotation angle formed by the four vertices $v_{2k},v_{2k+1},v_{2k+2},v_{2k+3}.$ Now, we let $\alpha_{2k}=\frac{\pi}{n}+c_k\delta$, $\alpha_{2k+1}=\frac{\pi}{n}-c_k\delta$ with $0<\delta<\frac{\pi}{n}$ and $c_k\in\{1,-1\}$ for all $k=0,1,\dots,\frac{n}{2}-1$.

Due to the symmetry with respect to the edge $v_0-v_1$, we have:

\begin{prop}\label{prop:sum}
$$\sum_{i=0}^{\frac{n}{2}-1}\alpha_i=\frac{\pi}{2},\eqno (1)$$
and
$$L(P_n)=\sum_{i=0}^{\frac{n}{2}-1}4\sin\frac{\alpha_i}{2},\eqno (2)$$
$$W(P_n)=\min_{j=0,1,\dots, \frac{n}{2}-1}\cos\frac{\alpha_j}{2}.\eqno (3)$$

\end{prop}

\begin{proof}
(1)(2)(3) come from the fact that $P_n$ is symmetric, and $\alpha_{2k}=\frac{\pi}{n}+c_k\delta$, $\alpha_{2k+1}=\frac{\pi}{n}-c_k\delta$ with $0<\delta<\frac{\pi}{n}$.
\end{proof}
We use cartesian coordinates to describe the $n-$gon $P_n$, assuming that a vertex $v_i$, $i=0,1,\dots,n-1$, is positioned at abscissa $x_i$ and ordinate $y_i$. Placing the vertex $v_1$ at the origin, we set $x_1=y_1=0.$ We also assume that $P_n$ is in the half plane $y\geq 0$.

\begin{prop}\label{prop:change}
If we place the vertex $v_0$ at $(0,1)$ in the plane, then the difference of $x-$coordinates between two vertices $v_{2k},v_{2k+2}$ of $P_n$ is

$$\Delta_x=
 \begin{cases}
 -\sin\left(2k\cdot\frac{\pi}{n}\right)+\sin\left(\left(2k+1\right)\cdot\frac{\pi}{n}+c_k\delta\right),&\text{if $b_k=1$},\\
 -\sin\left(\left(2k+1\right)\cdot\frac{\pi}{n}+c_k\delta\right)+\sin\left((2k+2)\cdot\frac{\pi}{n}\right),&\text{if $b_k=-1$.}
 \end{cases}$$
\end{prop}

\begin{proof}
Based on the construction of $P_n$, the side $v_{2k}-v_{2k+2}$ is formed by two ways. If the diameter changes in the way of $v_{2k}v_{2k+1}\rightarrow v_{2k+1}v_{2k+2}\rightarrow v_{2k+2}v_{2k+3}.$ Then
$$\Delta_x=-\sin(\sum_{i=0}^{2k-1}\alpha_i)+\sin(\sum_{i=0}^{2k}\alpha_i),$$
If the diameter changes in the way of $v_{2k}v_{2k+1}\rightarrow v_{2k}v_{2k+3}\rightarrow v_{2k+3}v_{2k+2}$, then
$$\Delta_x=-\sin(\sum_{i=0}^{2k}\alpha_i)+\sin(\sum_{i=0}^{2k+1}\alpha_i).$$
Also since $\alpha_{2k}=\frac{\pi}{n}+c_k\delta$, $\alpha_{2k+1}=\frac{\pi}{n}-c_k\delta$ with $0<\delta<\frac{\pi}{n}$, therefore we complete the proof of this proposition.
\end{proof}

Since we place the vertex $v_0$ at $(0,1)$ in the plane, then,
 $$x_{\frac{n}{2}}=-x_{\frac{n}{2}+1},\eqno (4a)$$
 $$y_{\frac{n}{2}}=y_{\frac{n}{2}+1}. \eqno (4b)$$
 Since the edge $v_{\frac{n}{2}}-v_{\frac{n}{2}+1}$ is horizontal and $\parallel v_{\frac{n}{2}}-v_{\frac{n}{2}+1}\parallel=1,$ we also have
 $$x_{\frac{n}{2}}=\frac{1}{2}=-x_{\frac{n}{2}+1}.\eqno (5)$$
 Then by Proposition \ref{prop:sum}, the perimeter (2) and width (3) become
 $$L(P_n)=2n\sin\frac{\pi}{2n}\cos\frac{\delta}{2},\eqno (6)$$
 $$W(P_n)=\cos\left(\frac{\pi}{2n}+\frac{\delta }{2}\right).\eqno(7)$$
Therefore, it is worth to consider the optimization problem: $\min\delta$.

\section{Proof of Theorem \ref{theo:main}}

\subsection{Solution to the optimization problem}

By Proposition \ref{prop:change}, the change of $x-$axis of between vertices $v_{2k},v_{2k+2}$ is
$$\Delta_x=-b_k\sin\left(\left(2k+1\right)\cdot\frac{\pi}{n}-b_k\frac{\pi}{n}\right)+b_k\sin\left(\left(2k+1\right)\cdot\frac{\pi}{n}+c_k\delta\right)\eqno (8)$$
By trigonometric function property,

\begin{align*}
\Delta_x=&-b_k\sin\left(\left(2k+1\right)\cdot\frac{\pi}{n}\right)\cos\left(\frac{\pi}{n}\right)+\cos\left(\left(2k+1\right)\cdot\frac{\pi}{n}\right)\sin\left(\frac{\pi}{n}\right)\\
&+b_k\sin\left(\left(2k+1\right)\cdot\frac{\pi}{n}\right)\cos\delta+b_kc_k\cos\left(\left(2k+1\right)\cdot\frac{\pi}{n}\right)\sin\delta.
\end{align*}
Therefore, coordinate $\left(x_{\frac{n}{2}},y_{\frac{n}{2}}\right)$ of $v_{\frac{n}{2}}$  in (5) are given by
\begin{align*}
x_{\frac{n}{2}}=\frac{1}{2}=&\sum_{k=0}^{\frac{n}{4}-1}\cos\left(\left(2k+1\right)\cdot\frac{\pi}{n}\right)\sin\left(\frac{\pi}{n}\right)\\
&+\sum_{k=0}^{\frac{n}{4}-1}b_k\sin\left(\left(2k+1\right)\cdot\frac{\pi}{n}\right)\left(\cos\delta-\cos\frac{\pi}{n}\right)\\
&+\sum_{k=0}^{\frac{n}{4}-1}b_kc_k\cos\left(\left(2k+1\right)\cdot\frac{\pi}{n}\right)\sin\delta.
\end{align*}

\begin{align*}
y_{\frac{n}{2}}=&\sum_{k=0}^{\frac{n}{4}-1}\sin\left(\left(2k+1\right)\cdot\frac{\pi}{n}\right)\sin\left(\frac{\pi}{n}\right)\\
&+\sum_{k=0}^{\frac{n}{4}-1}b_k\sin\left(\left(2k+1\right)\cdot\frac{\pi}{n}\right)\left(\sin\delta-\cos\frac{\pi}{n}\right)\\
&+\sum_{k=0}^{\frac{n}{4}-1}b_kc_k\sin\left(\left(2k+1\right)\cdot\frac{\pi}{n}\right)\cos\delta.
\end{align*}
We notice that
$$\sum_{k=0}^{\frac{n}{4}-1}\cos\left(\left(2k+1\right)\cdot\frac{\pi}{n}\right)\sin\left(\frac{\pi}{n}\right)=\frac{1}{2},$$
therefore,
$$\frac{\sin\delta}{\cos\delta-\cos\frac{\pi}{n}}=\frac{-\sum_{k=0}^{\frac{n}{4}-1}b_k\sin\left(\left(2k+1\right)\cdot\frac{\pi}{n}\right)}
{\sum_{k=0}^{\frac{n}{4}-1}b_kc_k\cos\left(\left(2k+1\right)\cdot\frac{\pi}{n}\right)}.$$
Since we assumed that $\delta\in(0,\pi)$, and $\frac{\sin x}{\cos x-\cos\frac{\pi}{n}}$ is increasing in $(0,\frac{\pi}{n})$, then we are looking for the minimum positive value of the right hand side, i.e.,
\begin{align*}
\sigma_n=&\min_{b_k=\pm 1}\mid\frac{\sum_{k=0}^{\frac{n}{4}-1}b_k\sin\left(\left(2k+1\right)\cdot\frac{\pi}{n}\right)}
{\sum_{k=0}^{\frac{n}{4}-1}b_kc_k\cos\left(\left(2k+1\right)\cdot\frac{\pi}{n}\right)}\mid\\
=&\min_{b_k=\pm 1}\frac{\mid\sum_{k=0}^{\frac{n}{4}-1}b_k\sin\left(\left(2k+1\right)\cdot\frac{\pi}{n}\right)\mid}{\sum_{k=0}^{\frac{n}{4}-1}\cos\left(\left(2k+1\right)
\cdot\frac{\pi}{n}\right)}\\
=&2\sin\frac{\pi}{n}\cdot \min_{b_k=\pm 1}\mid\sum_{k=0}^{\frac{n}{4}-1}b_k\sin\left(\left(2k+1\right)\cdot\frac{\pi}{n}\right)\mid.
\end{align*}
Let
$$M_n=\min_{b_k=\pm 1}\mid\sum_{k=0}^{\frac{n}{4}-1}b_k\sin\left((2k+1)\cdot\frac{\pi}{n}\right)\mid.$$
Therefore, we are looking forward for the estimation of $M_n$.

\begin{prop}
	Let $n=2^s\geq 4$, then $M_{4n}\leq M_n\cdot 4\sin\frac{\pi}{4n}\sin{\frac{\pi}{2n}}.$
\end{prop}

\begin{proof}
Let
$(b_{4k+1},b_{4k+3},b_{4k+5},b_{4k+7})=b_k^{'}(-1,1,1,-1).$
Then we have that
\begin{align*}
M_{4n}&=\min_{b_k=\pm 1}\mid\sum_{k=0}^{n-1}b_k\sin\frac{2k+1}{4n}\cdot\pi\mid\\
&\leq\min_{b_k^{'}=\pm 1}\mid\sum_{k=0}^{\frac{n}{4}-1}b_k^{'}\cdot \left(-\sin{\frac{8k+1}{4n}}\pi+\sin{\frac{8k+3}{4n}}\pi+\sin{\frac{8k+5}{4n}}\pi-\sin{\frac{8k+7}{4n}}\pi\right)\mid\\
&=\min_{b_k^{'}=\pm 1}\mid\sum_{k=0}^{\frac{n}{4}-1}b_k^{'}\cdot(2\sin\frac{\pi}{4n}\cos\frac{8k+2}{4n}\pi-2\sin\frac{\pi}{4n}\cos\frac{8k+6}{4n}\pi)\mid\\
&=\min_{b_k^{'}=\pm 1}\mid\sum_{k=0}^{\frac{n}{4}-1}b_k^{'}\cdot4\sin\frac{\pi}{4n}\sin{\frac{2\pi}{4n}}\sin\frac{8k+4}{4n}\pi\mid\\
&=\min_{b_k^{'}=\pm 1}\mid\sum_{k=0}^{\frac{n}{4}-1}b_k^{'}\cdot\sin\frac{8k+4}{4n}\pi\mid \cdot 4\sin\frac{\pi}{4n}\sin{\frac{2\pi}{4n}}\\
&=M_n\cdot 4\sin\frac{\pi}{4n}\sin{\frac{\pi}{2n}}.
\end{align*}
\end{proof}
By easy calculation, we have
$$M_4=\sin\frac{\pi}{4},$$
$$M_8=-\sin\frac{\pi}{8}+\sin\frac{3\pi}{8}=2\sin\frac{\pi}{8}\cos\frac{\pi}{4}=2\sin\frac{\pi}{4}\sin\frac{\pi}{8}.$$
By induction, we have
$$M_{2^s}\leq 2^{s-2}\prod_{k=1}^{s}\sin{\frac{\pi}{2^{k}}}.$$
Moreover, since $\sin x\leq x$, for $s\geq 2$, we have
$$M_{2^s}\leq 2^{s-2}\cdot\frac{1}{\sqrt{2}}\prod_{k=3}^{s}\frac{\pi}{2^k}= 2^{s-5/2}\frac{\pi^{s-2}}{2^{(3+s)(s-2)/2}}=\frac{\pi^{s-2}}{2^{(s^2-s-1)/2}}.$$
Take $n=2^s$, we have
$$M_n\leq \frac{\sqrt{2}}{\pi^2}\cdot\frac{\pi^{\log_{2}{n}}}{n^{\frac{\log_{2}{n}-1}{2}}}.$$
That is,
\begin{prop}\label{prop:esti}
	$$M_n\leq \frac{\sqrt{2}}{\pi^2}\cdot\frac{\pi^{\log_{2}{n}}}{n^{\frac{\log_{2}{n}-1}{2}}}.$$
\end{prop}
Now we prove Theorem \ref{theo:main}.

\begin{proof}[Proof of Theorem \ref{theo:main}]
By Proposition \ref{prop:esti}, let
$$\sigma_n=\min_{b_k=\pm 1}\mid\frac{\sum_{k=0}^{\frac{n}{4}-1}b_k\sin\left(\left(2k+1\right)\cdot\frac{\pi}{n}\right)}
{\sum_{k=0}^{\frac{n}{4}-1}b_kc_k\cos\left(\left(2k+1\right)\cdot\frac{\pi}{n}\right)}\mid,$$
and let $\delta_n$ be the solution of
$$\sigma_n\left(\cos x-\cos\frac{\pi}{n}\right)=\sin x.$$
This equation has a solution $\delta_n$ satisfying
$$\delta_n=\arccos\left(\frac{\sigma_n}{\sqrt{1+\sigma_n^2}}\cos{\frac{\pi}{n}}\right)-\arccos\frac{\sigma_n}{\sqrt{1+\sigma_n^2}}.$$
Since $n$ is large, $\cos\delta_n\rightarrow 1$, $\sin\delta_n\rightarrow \delta_n$, then
$$\delta_n\leq (1-\cos\frac{\pi}{n})\sigma_n=2\sin^2\left(\frac{\pi}{2n}\right)\sigma_n.$$
And $\sigma_n$ is sufficiently smaller than $\frac{\pi}{n}$, we have
$$\delta_n\leq O\left(\frac{\pi^2}{2n^2}\right)\sigma_n\leq O\left(\frac{\pi^3}{n^3}\right)M_n\leq O\left( \frac{\pi^{\log_{2}{n}}}{n^{\frac{\log_{2}{n}+5}{2}}}\right).$$

\begin{figure}[htbp]
\begin{minipage}[htbp]{0.55\linewidth}
\centering
\includegraphics[height=6cm,width=6cm]{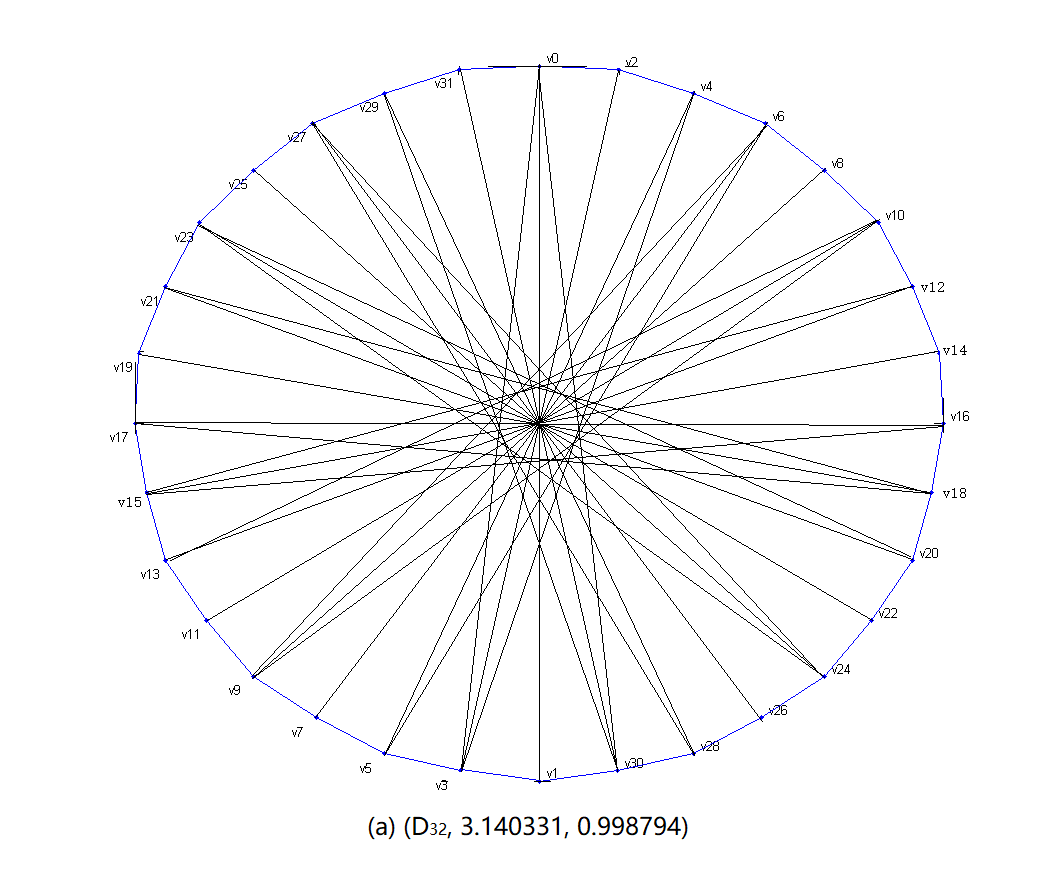}
\end{minipage}%
\hfill
\begin{minipage}[htbp]{0.55\linewidth}
\centering
\includegraphics[height=6cm,width=6cm]{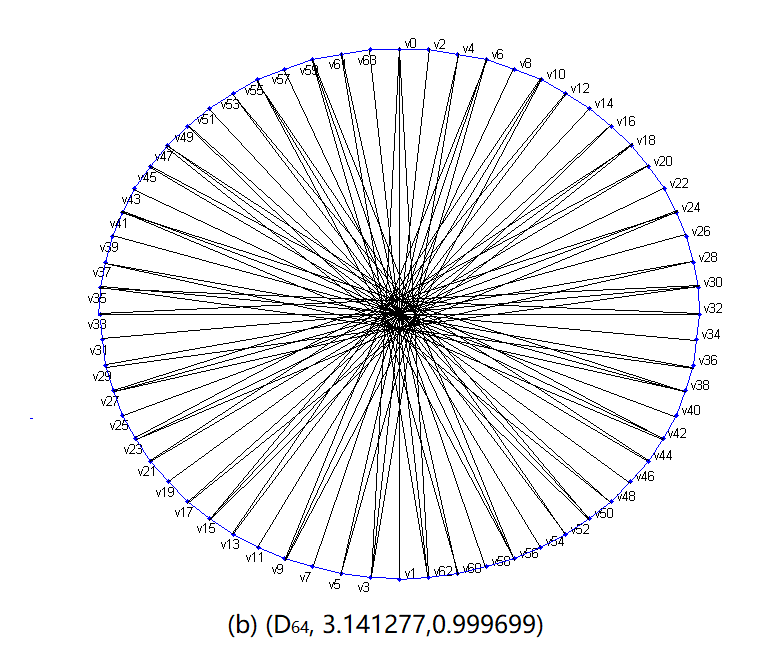}
\end{minipage}
\caption{Polygons $(D_n, L(D_n), W(D_n))$ defined in Theorem \ref{theo:main}: (a) $32$-gon; (b) $64$-gon.}
\end{figure}
Due to Section \ref{section}, Proposition \ref{prop:sum} and \ref{prop:change}, there exists a convex small $n$-gon $D_n$ obtained by setting $\delta=\delta_n$ such that
   $$L(D_n)=2n\sin\frac{\pi}{2n}\cos\left(\frac{1}{2}\arccos\left(\frac{\sigma_n}{\sqrt{1+\sigma_n^2}}\cos{\frac{\pi}{n}}\right)-\frac{1}{2}\arccos\frac{\sigma_n}{\sqrt{1+\sigma_n^2}}\right),
   $$
  and
$$\overline{L}_n-L(D_n)\leq O\left(\frac{\pi^{\log_2n^2}}{n^{\log_2n+5}}\right).$$
For the width part, we just calculate $W(D_n)$, i.e.,
 $$W(D_n)=\cos\left(\frac{\pi}{2n}+\frac{1}{2}\arccos\left(\frac{\sigma_n}{\sqrt{1+\sigma_n^2}}\cos{\frac{\pi}{n}}\right)-\frac{1}{2}\arccos\frac{\sigma_n}{\sqrt{1+\sigma_n^2}}\right),$$
and
$$\overline{W}_n-W(D_n)\leq O\left(\frac{\pi^{\log_2n^2}}{n^{\log_2n+7}}\right).$$
\end{proof}

\subsection{Some numerical results}

Due to Theorem \ref{theo:main} and Section \ref{section}, we illustrate $D_n$ for some $n$ in Figure $6$.
\begin{table}[htb]
\small
\centering
\caption{Perimeters of $D_n$}
\label{table1}
\begin{tabular}{ccccc}
\hline
$n$&$L(R_n)$&$L(C_n)$&$L(D_n)$&$\overline{L}_n$\\
\hline
16  &3.121445152258052&3.136547508015487&$\geq$3.136547508015487&3.136548490545939\\
32  &3.136548490545939&3.140331153461366&$\geq$3.140331156355381&3.140331156954753\\
64  &3.140331156954753&3.141277250919435&$\geq$3.141277250932682&3.141277250932773\\
128 &3.141277250932773&3.141513801144249&$\geq$3.141513801144301&3.141513801144301\\
\hline
\end{tabular}
\end{table}

\begin{table}[htb]
\centering
\caption{Widths of $D_n$}
\label{table2}
\begin{tabular}{ccccc}
\hline
$n$&$W(R_n)$&$W(C_n)$&$W(D_n)$&$\overline{W}_n$\\
\hline
16  &0.980785280403230&0.995106832387674&$\geq$0.995106832387674&0.995184726672197\\
32  &0.995184726672197&0.998793140652984&$\geq$0.998794497340913&0.998795456205172\\
64  &0.998795456205172&0.999698747175479&$\geq$0.999698812803775&0.999698818696204\\
128 &0.999698818696204&0.999924699610472&$\geq$0.999924701821059&0.999924701839145\\
\hline
\end{tabular}
\end{table}

Table \ref{table1} shows the perimeters of $D_n$, along with the upper bounds $\overline{L}_n$, the perimeters of polygons $R_n$ and $C_n$. 

Table \ref{table2} shows the widths of $D_n$, along with the upper bounds $\overline{W}_n$, the widths of $R_n$ and $C_n$.

As a remark, one can furtherly improve the value by changing $c_k\in\{1,-1\}$ with $c_k\in[-1,1]$.

\section*{Acknowledge}
The second author, the third author and the fourth author are supported by the National Nature Science Foundation of China
(NSFC 11921001) and the National Key Research and Development Program of China (2018YFA0704701).

\end{document}